\documentclass{elsarticle}

\usepackage{lineno,hyperref}

\usepackage{graphicx}
\usepackage{textcomp}
\usepackage{amssymb}
\usepackage{epstopdf}
\usepackage{amsmath,amsthm,amscd,amssymb}
\usepackage{latexsym}
\usepackage{bm}
\usepackage{geometry}                

\makeatletter
\def\ps@pprintTitle{%
 \let\@oddhead\@empty
 \let\@evenhead\@empty
 \def\@oddfoot{\centerline{\thepage}}%
 \let\@evenfoot\@oddfoot}
\makeatother

\journal{European Journal of Combinatorics}

\numberwithin{equation}{section}

\theoremstyle{plain}
\newtheorem{theorem}{Theorem}[section]
\newtheorem{lemma}[theorem]{Lemma}
\newtheorem{corollary}[theorem]{Corollary}

\theoremstyle{definition}
\newtheorem{definition}[theorem]{Definition}

\newtheorem{case[theorem]}{Case}

\theoremstyle{remark}
\newtheorem{remark}[theorem]{Remark}

\numberwithin{equation}{section}

\def\bb #1{ {\mathbb #1} }
\def\ds{\displaystyle}

\begin{document}

\begin{frontmatter}

\title{Occurrence of Right Angles in Vector Spaces Over Finite Fields}

\author{Michael Bennett}
\address{7208 Overhill Dr.\\ Pomona, NY 10970\\ \,\,\, mike.b.bennett87@gmail.com}

\begin{abstract} Here we examine some Erd\H os-Falconer-type problems in vector spaces over finite fields involving right angles. Our main goals are to show that

a) a subset $A \subset \bb{F}_q^d$ of size $\geq cq^{\frac{d+2}{3}}$ contains three points which generate a right angle, and

b) a subset $A \subset \bb{F}_q^d$ of size $\geq Cq^{\frac{d+2}{2}}$ contains two points which generate a right angle with

\hspace{.115 in} the vertex at the origin.

\noindent
We will also prove that b) is sharp up to constants and provide some partial results for similar problems related to spread and collinear triples.
\end{abstract}

\end{frontmatter}


\section{Introduction}

In this paper, we will be looking at some extremal problems in combinatorial geometry similar to those proposed by Erd\H os and Falconer, but in the setting of finite fields rather than in Euclidean space. Throughout, $q$ will be a power of an odd prime, and we will denote the field with $q$ elements by $\bb{F}_q$. We would like to determine how large a subset of $\bb{F}_q^d$ (with $q$ much larger than $d$) needs to be to guarantee that the set contains three points which form a right angle. We will also make some observations about extremal problems regarding ``spread." Spread is essentially a finite field analogue of the Euclidean angle. When we say that a triple $(a,b,c)$ ``generates" an angle (or spread) $\theta$, we mean that the angle between the vectors $a-b$ and $c-b$ is $\theta$. This is, in a way, an extension of the right angle problem, because to say that a triple ``generates a right angle" is roughly equivalent to saying that the spread generated by the triple is $1$. These notions will be explained much more thoroughly in section \ref{spread}. Finally, we will ask how large a subset of $\bb{F}_q^d$ needs to be to guarantee that three points in that subset are collinear and provide relatively trivial bounds on this problem, not coming particularly close to a decisive answer to the question.


This paper was motivated by a paper of Harangi et. al. which looked at the same problems, as well as many others, in $\bb{R}^d$. By looking at the same problems in a different setting, we hope to improve our understanding of the relation between the geometric structure of $\bb{R}^d$ and $\bb{F}_q^d$.

\vskip.125in

\section{Main Theorems}


\begin{definition}
We say that an ordered triple of points $(x,y,z) \in \bb{F}_q^d \times \bb{F}_q^d \times \bb{F}_q^d$ forms a right angle if $x,y,$ and $z$ are distinct, and the vectors $x-y$ and $z-y$ have dot product 0.
\end{definition}


\begin{theorem}\label{main1}
If $A,B \subset \mathbb{F}_q^d$, $A \cap B = \emptyset$, and $|A|^2|B| \geq 4q^{d+2}$, then there are $x,y \in A$ and $z \in B$ so that $(x,y,z)$ forms a right angle.
\end{theorem}

Following immediately from the theorem, we have:

\begin{corollary}\label{singleset}
If $A \subset \mathbb{F}_q^d$ and $|A| \geq 4q^\frac{d+2}{3}$ then $A$ contains a right angle.
\end{corollary}

\vspace{.1 in}

That is, given a set $A$, we may decompose it into disjoint sets $A'$ and $B'$ of roughly equal size and then apply the theorem. We will see that the bound from \ref{main1} is best possible up to the implied constant. It is not known whether the result is improvable in the form of \ref{singleset}.

We will also examine the problem of existence of right angles whose vertex is fixed at the origin:

\vspace{.1 in}

\begin{theorem}\label{main2} If $A \subset \mathbb{F}_q^d$ and $|A| \geq 4q^\frac{d+2}{2}$, then there are $x,y \in A$ so that $x \cdot y = 0$. Moreover, if $q$ is a prime, then there exists a subset $B \subset \bb{F}_q^d$ so that $|B| = \Omega(q^\frac{d+2}{2})$ but $\{x,y \in B: x \cdot y = 0\} = \emptyset$.
\end{theorem}

\section{Comparison with Euclidean Analog}

Perhaps the most interesting aspect of these results are how they differ from those found in \cite{HKKMMMS}. In their paper, Harangi et al. solve Euclidean versions of these problems. There they pose the problem in two different ways: For a given angle, how large must a set in $\bb{R}^d$ be to guarantee that

\vspace{.2in}
\noindent
1) that angle is generated by that set?

\vspace{.1in}
\noindent
2) an angle with size within $\delta$ of the given angle is generated by that set?

\vspace{.2in}
When we say ``how large," we are referring to the Hausdorff dimension of the set. In the paper they discover that $90^\circ$ angles are a singular case. In the sense of question 1, the critical dimension for $90^\circ$ angles is between $\frac{d}{2}$ and $\frac{d+1}{2}$. That is, there are examples of sets of Hausdorff dimension arbitrarily close to $\frac{d}{2}$ which contain no $90^\circ$ angles, while any compact set with Hausdorff dimension larger than $\frac{d+1}{2}$ \textit{must} contain a $90^\circ$ angle (which they prove rather concisely in the paper). For 0 and $180^\circ$, the critical dimension is $d-1$ (because of the $(d-1)$-sphere), while for other angles only partial answers exist. Question 2 is completely resolved in this paper, where they find that a set of Hausdorff dimension 1 always has an angle very close to $90^\circ$, while a set of lower dimension need not. It is known that any sufficiently large finite set of points contains three points at an angle as close to $180^\circ$ as desired (and therefore close to $0^\circ$ as well) (see \cite{EF}). Interestingly, $60^\circ$ and $120^\circ$ angles also have a separate result  that depends on $\delta$. For any other given angle, however, a set of Hausdorff dimension that increases to $\infty$ as $d$ goes to $\infty$ can be constructed so that a neighborhood of that angle is avoided.

In Euclidean space, Hausdorff dimension is a natural way to classify the dimension of an arbitrary set. The best analog to this classification in vector spaces over finite fields is to think of the dimension of a set $S$ as $\approx\log_q(|S|)$. This is why, in the finite field setting, the quantity we are most interested in is the exponent $\alpha$ in $|S| = cq^\alpha$. In extremal geometry, problems that can be solved in Euclidean space can often be solved in a similar manner over finite fields. Consequently, one might expect the exponent in the finite field setting to match the Hausdorff dimension in the Euclidean setting. This is why the case of right angles is interesting here. We find that the critical dimension in vector spaces over finite fields is bounded above by $\frac{d+2}{3}$, which is significantly different from the (question 1) threshold found in \cite{HKKMMMS}. As far as the Euclidean analog of theorem \ref{main1}, consider the set $(0, \infty) \times \ldots \times (0, \infty) \subset \bb{R}^d$. The dot product of any two elements in this set must be positive, and the set clearly has Hausdorff dimension $d$, so this also presents a discrepancy between the continuous and finite settings.


\section{Proof of Theorem \ref{main1}}

We begin with the following lemma.

\begin{lemma} \label{lemma}
Let $H$ be a nonempty set of hyperplanes in $\bb{F}_q^d$. Let $E = \bigcup_{h \in H} h$. Then $|E^c| \leq \frac{q^{d+1}}{|H|}.$
\end{lemma}

\begin{proof}

We will use `$h$' as the characteristic function for the set $h$. First we note that
\begin{equation}\label{easysum}
\sum_{h \in H, x \in \bb{F}_q^d} h(x) =  \sum_{h \in H} |h| = q^{d-1}|H|.
\end{equation}
Using Cauchy-Schwarz and the fact that every $h \in H$ is a subset of $E$, we also have
\begin{align*}
\sum_{h \in H, x \in \bb{F}_q^d} h(x) = \sum_{h \in H, x \in E} h(x) &\leq \left(\sum_{x \in E} 1^2 \right)^{1/2} \left(\sum_{x \in E} \left( \sum_{h \in H} h(x) \right)^2 \right)^{1/2}\\
&= |E|^{1/2} \left( \sum_{h_1,h_2 \in H} \sum_{x \in \bb{F}_q^d} h_1(x)h_2(x) \right)^{1/2}.
\end{align*}
If $h_1 = h_2$, then $\sum_x h_1(x)h_2(x) = q^{d-1}$, and otherwise $\sum_x h_1(x)h_2(x) \leq q^{d-2}$. Making use of equation \ref{easysum}, we now have
$$
q^{d-1}|H| \leq |E|^{1/2} \left( q^{d-1}|H| + q^{d-2}|H|^2 \right)^{1/2}.
$$
Solving for $|E|$ gives
$$
|E| \geq \frac{q^{2d-2}|H|^2}{q^{d-1}|H| + q^{d-2}|H|^2} = \frac{q^d|H|}{q + |H|}.
$$
Lastly,
$$
|E^c| = q^d-|E| \leq \frac{q^{d+1}}{q + |H|} \leq \frac{q^{d+1}}{|H|}.
$$

\end{proof}

We are now ready to prove the theorem. For any ordered pair $(x,y) \in \bb{F}_q^d \times \bb{F}_q^d$ with $x \neq y$, Define $h_{xy}$ to be
$$\{z: (z-x) \cdot (y-x) = 0\},$$
i.e. the (translated) hyperplane through the point $x$ with normal vector in the direction of $y-x$. 

\vspace{.1in}
\noindent
Let $A,B \subset \bb{F}_q^d$ with $A \cap B = \emptyset$. Let $A'$ be any subset of $A$ with $|A'| = \left \lfloor \frac{|A|}{2} \right \rfloor$. We consider the set of hyperplanes $H = \{h_{xy} : (x,y) \in A' \times B, x \neq y \}$. Notice first, that if $c \in \left(h_{ab} \backslash \{a\}\right) \cap A$, then $(b,a,c)$ defines a right angle. Thus we will assume that $h_{ab} \cap A = \{a\}$ for all $a \in A, b \in B$. Consequently, $h_{ab} = h_{cd}$ only when $a = c$ (as otherwise $h_{ab} \cap A$ contains two distinct elements $a$ and $c$). It is also important to note that under certain circumstances, it is possible to have $y \in h_{xy}$ (see \ref{iso}). We want to ensure that $x$ is the only point of $A$ in $h_{xy}$, which is why we insist that $A$ and $B$ be disjoint. Finally, notice that for fixed $a$ and $b$, there are exactly $q-1$ choices of $b'$ so that $h_{ab} = h_{ab'}$, i.e. all points on the line through $a$ and $b$ except for $a$.

\vspace{.1in}
\noindent
Because there are $|A'||B|$ elements in $A' \times B$, there must be at least $\frac{|A'||B|}{q-1}$ distinct planes in $H$.

Let $E = \bigcup_{h \in H} h$. By assumption, $A \backslash A' \subset E^c$. Notice that $|A \backslash A'| = \left \lceil \frac{|A|}{2} \right\rceil$. Using Lemma \ref{lemma}, we have
$$|A \backslash A'| \leq \frac{q^{d+1}(q-1)}{|A'||B|}$$
or equivalently,
$$\left \lceil\frac{|A|}{2} \right\rceil  \leq \frac{q^{d+1}(q-1)}{\left\lfloor \frac{|A|}{2} \right\rfloor |B|}$$
Solving for $|A|$ and $|B|$ gives us  
$$|A|^2|B| \leq 4q^{d+2}$$
In other words, if none of the points of $B$ completes a right angle with any of the pairs of points in $A$, then $|A|^2|B|$ can be no larger than $4q^{d+2}$.

\subsection{\textbf{Sharpness}}
If $d \geq 3$ then we can find an isotropic line $L \subset \bb{F}_q^d$ (see \ref{iso}). Indeed, if $x$ is an isotropic vector then the line through $x$ and the origin has the aforementioned property.

Now let $L$ be an isotropic line and let $H$ be the hyperplane that is perpendicular to $L$. Remember, because $L$ is isotropic, $L \subset H$. Because of this, the only hyperplane parallel to $H$ that meets $L$ is $H$ itself. Then, in the language of theorem \ref{main1}, we choose $A = L$ and $B = \bb{F}_q^d \backslash H$. Notice that $|A|^2|B| = q^{d+2}(1 + o(1))$. If $x,y \in A$ and $z \in B$, then
$$
(x-y) \cdot (z-y) = z \cdot (x-y) - x \cdot y + y \cdot y = z \cdot (x-y).
$$
Now $x-y \in L$ and $z$ is not orthogonal to $L$ by construction, so $(x-y) \cdot (z-y) \neq 0$, which completes the proof.
In fact, this setup also works for $d=2$ whenever $q \equiv 1 \bmod 4$, since $\bb{F}_q$ contains a square root of $-1$, and hence an isotropic line. (In dimension $2$, notice that an isotropic line is its own orthogonal hyperplane).

\section{Proof of Theorem \ref{main2}}

This theorem has already been proved by Hart et. al. in \cite{HIKR11} using Fourier techniques. In fact they also show that if $|A| \geq Cq^\frac{d+1}{2}$, then for any \textit{nonzero} scalar $t$, there are two points in $A$ (not necessarily distinct) whose dot product is $t$. Here we provide a different short proof of the result for $t = 0$, once again using \ref{lemma}. More importantly, we also show that the $0$ case is an exceptional case---that a guaranteed dot product of $0$ requires about $q^\frac{d+2}{2}$ points of $\bb{F}_q^d$, at least in the case $q$ is prime.

Let $A \subset \bb{F}_q^d$. If $A$ contains the origin, just use $A \backslash \{\vec{0}\}$ instead. Let $h_a$ be the hyperplane (through the origin) perpendicular to $a$. Notice that for each $a \neq \vec{0}$ there are exactly $q-1$ points $b \in \bb{F}_q^d \backslash \{\vec{0}\}$ so that $h_a = h_b$, namely $b = sa$ for any scalar $s \neq 0$. Therefore, if we let $E = \bigcup_{a \in A} h_a$, $E$ is the union of at least $\frac{|A|}{q-1}$ distinct hyperplanes.

By lemma \ref{lemma}, $|E^c| \leq \frac{q^{d+1}}{|A|/(q-1)}$. If $|A| > |E^c|$, then $A$ and $E$ intersect, and thus two points of $A$ have dot product $0$. To avoid this, we must have $|A| < \frac{q^{d+1}}{|A|/q}$, or equivalently, $|A| < q^{\frac{d+2}{2}}$.

\vspace{.2 in}
Any isotropic vector is orthogonal to itself. If we want to be sure that there are two \textit{distinct} points $a$ and $b$ so that $(a,\vec{0},b)$ forms a right angle, then we can use the same trick from the proof of \ref{main1}. That is, we split $A$ into two disjoint subsets of size around $|A|/2$, generate a set of hyperplanes from one of the subsets, and show that the other subset intersects one of those hyperplanes. In this case, we get $|A| < 4q^{\frac{d+2}{2}}$.

\vspace{.2 in}



\subsection{\textbf{Sharpness}}

We should first note that in the case $d=3$, the sharpness of this result has essentially been proved already by Mubayi and Williford. In \cite{MW}, they find a set of mutually non-orthogonal lines through the origin in $\bb{F}_q^3$ of cardinality $\Theta(q^{3/2})$. In fact, they construct explicit sets of lines for every $q = p^n$ (separating $p=2$ and $p \neq 2$ cases, as well as cases of even and odd $n$) and find tight bounds on the sizes of these sets. Since all of these lines pass through the origin, the set of all points lying on these lines has size $\Theta(q^{5/2})$. By construction, no pair of these points can have dot product zero, and $\Theta(q^{5/2})$ is sharp per \ref{main2}. (Mubayi and Williford also prove that their estimate is sharp using graph theoretic techniques).

Here we will show that for any $d$, there is a set of mutually non-orthogonal lines through the origin in $\bb{F}_p^d$ of size $\Theta(p^{\frac{d}{2}})$. This will in turn give us a set of $\Theta(p^{\frac{d+2}{2}})$ points, showing that our estimate in \ref{main2} is sharp up to a factor depending on $d$. Notice, however, that we insist on a prime $p$, as our method of proof calls for it.

\vspace{.1 in}
Order the elements of $\bb{F}_p$ in the natural way, from $0$ to $p-1$. Let $d \geq 2$ (but much smaller than $p$) and define $\sigma = \lfloor \sqrt{p/d} \rfloor$ for ease of notation. Notice that $\sigma \neq \sqrt{p/d}$, as $ d \nmid p$.

\noindent
Let $A = \{a \in \bb{F}_p: 1 \leq a \leq \sigma\}$ and let $A^d \subset \bb{F}_p^d$ be the $d$-fold cross product. If $x,y \in A^d$ then $d \leq x \cdot y \leq p-1$. Right away, we see the set $A^d$ contains no pair of points whose dot product is $0$. We will show that the number of lines through the origin that meet $A^d$ is $\Theta(p^\frac{d}{2})$.

This is very similar to the problem of determining the number of lattice points in $\{1, \ldots, \sigma\}^d \subset \bb{Z}^d$ that are visible from the origin, or equivalently, the number of relatively prime $d$-tuples of integers between 1 and $\sigma$. This is well known to be about $\frac{\sigma^d}{\zeta(d)}$ for large $\sigma$ (see, for instance, \cite{Ny}), where $\zeta$ is the Riemann zeta function. Even though the notion of ``visibility from the origin" and ``relatively prime" are ambiguous in finite fields, this estimate will still work for us in this context. 

First, let us consider the case $d = 2$. The points $(a_1,a_2),(b_1,b_2) \in A^2$ lie on the same line if and only if $a_1b_2 = a_2b_1 \bmod p$. Because $1 \leq a_1,a_2,b_1,b_2 \leq \sigma$, we see $a_1b_2$ and $a_2b_1$ are both less than $p/d$, and therefore $a_1b_2 = a_2b_1$ (as elements of $\bb{Z}$). In other words, $(a_1,a_2),(b_1,b_2) \in A^2$ lie on the same line through the origin only if $(a_1,a_2),(b_1,b_2) \in \bb{Z}^2$ lie on the same line through the origin. It is easy to see that this extends to higher dimensions as well. Thus the number of distinct lines that meet $A^d$ is equal to the number of distinct lines that meet $\{1, \ldots, \sigma\}^d \subset \bb{Z}^d$.

In our case, the number of distinct lines through the origin that intersect $A^d$ must be
$$\frac{\sigma^d}{\zeta(d)}(1 - o(1)) = \frac{p^{\frac{d}{2}}}{d^{\frac{d}{2}}\zeta(d)}(1 - o(1)),$$ which gives us a set of $\ds \frac{1 - o(1)}{d^{\frac{d}{2}}\zeta(d)}p^{\frac{d+2}{2}}$ points in $\bb{F}_p^d$ so that no two have a dot product of 0.

\section{Spread}\label{spread}

The length of a vector $a \in \bb{F}_q^d$ is defined by

$$
\|a\| = \sum_{i = 1}^d a_i^2.
$$
We will use $\|a-b\|$ to refer to the distance between points $a$ and $b$. Without the square root, this definition does not agree with the Euclidean space definition, and doesn't fit the the qualifications of a metric; however it does retain the important property of invariance under the action of the orthogonal group. 

\vspace{.1 in}
The spread $S$ between two vectors $a,b \in \bb{F}_q^d$ is defined by
$$
S(a,b) = 1- \frac{(a \cdot b)^2}{\|a\| \|b\|}.
$$
We see that this definition is somewhat consistent with that of angles in Euclidean space, as 
$$
\sin^2(\theta) = 1-\frac{(a \cdot b)^2}{|a|^2 |b|^2}.
$$
It might seem more natural to define spread in analogy with the $\sin \theta$ or $\cos \theta$ formula, rather than $\sin ^2 \theta$. However, the square root operation is not well-defined from $\bb{F}_q \rightarrow \bb{F}_q$, so this is essentially impossible. If we were to define angles in {\it Euclidean} space via $\theta = \phi$ iff $\sin ^2 \theta = \sin^2\phi$, there would be no distinction between an angle and its supplement. But in vector spaces over finite fields, two such angles cannot be distinguished, because rays and lines are identical objects. That is, a ray originating at the origin through the point $b$ also goes through the point $-b$. So spread is a perfectly reasonable analog to the Euclidean angle.

\begin{definition}\label{iso}
An isotropic vector is a vector $\vec{v}$ such that $\vec{v} \neq \vec{0}$ but $\|\vec{v}\| = 0$. An isotropic line is a line through the origin in the direction of an isotropic vector.
\end{definition}

Notice that for any two vectors $a,b$ on an isotropic line, $a \cdot b = 0$.

\vspace{.1 in}

Our definition of spread is a bit problematic with isotropic vectors since we need to divide by vector length to determine the spread. We will just say the spread is undefined in this situation. Isotropic vectors can be found in $\bb{F}_q^d$ for $d > 2$ and in $\bb{F}_q^2$ when $q \equiv 1 \bmod 4$. Fortunately, in the arguments we present here, isotropic vectors are not a significant obstacle.

\noindent
Here are a few easily verified properties of spread:
\\

\noindent
$S(a,b) = S(ra,sb)$ for any $r,s \in \bb{F}_q^*.$

\noindent
$S(a,b) = S(b,a).$

\noindent
$S(a,b) = S(\sigma (a), \sigma (b))$ where $\sigma$ is an element of the orthogonal group.

\vspace{.2 in}
\noindent
These are consistent with properties of the angle between vectors in Euclidean space. Notice that two vectors with a spread of $1$ have a dot product of $0$, and the converse is true for non-isotropic vectors. Therefore, we can think of two vectors that have a spread of $1$ as making a right angle. In the next theorem we look at other spreads, which have surprisingly different results from spread $1$.

\vspace{.2 in}
We have the following theorem from \cite{HIKR11}, where $S_t^{d-1} = \{x \in \bb{F}_q^d: \|x\| = t\}$ is the $(d-1)$-dimensional sphere of radius $t$ centered at the origin:
\begin{theorem}\label{distsphere}
Suppose $E \subset S_t^{d-1} \subset \bb{F}_q^d$ and $|E| \geq cq^\frac{d}{2}$ for $t \neq 0$ and $d \geq 3$. Then the points of $E$ generate all distances if $d$ is even and generate a positive proportion of distances if $d$ is odd.
\end{theorem}

We also have the related Lemma 2.1 from \cite{IR}:

\begin{lemma}\label{zerodistance}
If $|E| \geq Cq^{d/2}$ for a sufficiently large $C$, then 
$$
\#\{(x,y) \in E \times E: \|x-y\| = 0\} \leq \frac{1}{\sqrt{2}}|E|^2.
$$

\end{lemma}

\noindent
In particular, this means that if $E$ is a large enough set, then a positive proportion of pairs of vertices in $E$ are separated by a nonzero distance.

We will use these results to make the proof of our next theorem rather easy.

\begin{theorem}\label{anglesphere}
Let $d \geq 3$, $E \subset \bb{F}_q^d$, $|E| \geq Cq^\frac{d+2}{2}$ for some sufficiently large constant $C$. Then there is a constant $c$ independent of $q$ so that at least $cq$ distinct spreads occur in $E$. 
\end{theorem}

\noindent
(The result from this theorem has been improved in \cite{LPV}. They show that $|E| \geq (1 + \epsilon)q^{\lceil d/2 \rceil}$ is sufficient to guarantee $cq$ spreads and that this estimate is tight. Because their paper references theorem \ref{anglesphere}, I have chosen to leave it in.)

\vspace{.1 in}

\begin{remark}
Notice that if $a,b \in \bb{F}_q^d$ and $\|a\| = \|b\|$, then $1- S(a,b) = 0$ or is a square in $\bb{F}_q$. Therefore, if $S^{d-1}$ is the $(d-1)$-dimensional sphere of radius $t$ in $\bb{F}_q^d$, then $\{1- S(a,b): a,b \in S_t^{d-1}\}$ contains no nonsquares. If we were to look at the set of spreads generated by triples of points on $S_t^{d-1}$, we may very well get all spreads. However, in our proof of the above theorem, our strategy will be to look at spreads of triples $(a,\vec{0},b)$, and thus we will only guarantee a positive proportion of spreads.
\end{remark}

\begin{proof}

Let $E \subset \bb{F}_q^d$ with $|E| > C_1q^\frac{d+2}{2}$. By lemma \ref{zerodistance}, at least $C_2|E|^2$ pairs of points in $E$ are separated by a nonzero distance. Therefore, at least one point $x \in E$ has the property $$\#\{y \in E: \|x-y\| \neq 0\} \geq C_3q^\frac{d+2}{2}.$$ Without loss of generality, we will suppose that $\vec{0}$ is such a point. With $C_3$ sufficiently large, there must be some $t \in \bb{F}_q^*$ so that the size of the sphere $S_t^{d-1}$ of radius $t$ centered at the origin has size $> C_4q^\frac{d}{2}$ by pigeonholing. For simplicity, we will suppose $t = 1$, but the following arguments work for any fixed value of $t$.

Next we will show that for points $a,b,c,d$ on the sphere,
$$
S(a,b) = S(c,d) \text{ if and only if } \|a-b\| = \|c\pm d\|.
$$
First, notice that $\|a-b\| = 2 - 2a \cdot b$. If $\|a-b\| = \|c-d\|$, then $a \cdot b = c \cdot d$. If $\|a-b\| = \|c-(-d)\|$, then $a \cdot b = -c \cdot d$. Either way, we have  $(a \cdot b)^2 = (c \cdot d)^2$, and thus $S(a,b) = S(c,d)$. This argument can be reversed to recover the other direction.

We now know that, given a positive proportion of the distances on the sphere, we generate a positive proportion of the spreads on the sphere. Invoking \ref{distsphere}, since $|E \cap S_t^{d-1}| \geq C_4q^\frac{d}{2}$ then the set of triples $\{(a,\vec{0},b): a,b \in E \cap S_t^{d-1}\}$ generate a positive proportion of spreads.

\end{proof}

This result says nothing of substance about  dimension $2$, so we will now prove the following:

\begin{theorem}\label{2D}
If $A \subset \bb{F}_q^2$ and $|A| \geq 2q-1$ then $A$ generates all spreads.
\end{theorem}

\noindent
This is, of course, sharp up to the coefficient of $q$, as a line in $\bb{F}_q^2$ has $q$ points but only generates one spread (namely spread = 0, as long as the line is not parallel to an isotropic line).

Before we start the proof, we need to make an observation. The ``spread between two lines," where we think of the vertex as the intersection point of the lines, is well-defined because

\vspace{.125 in}
\noindent
1) $S(a,b) = S(ra,sb)$ for any $r,s \in \bb{F}_q^*$ (as stated earlier) and 2) the spread generated by the triple $(a,b,c)$ is the same as the spread generated by $(a+z,b+z,c+z)$ for any $z \in \bb{F}_q^2$.

\vspace{.125 in}
\noindent
Property 1 shows that if $b$ is the intersection of two lines, then the spread generated by $(a,b,c)$ will be consistent for any choice of $a \neq b$ on the first line, and any choice of $c \neq b$ on the second. Property 2 shows that the point of intersection need not be the origin. Moreover, it is easy to see that if $l_1$ and $l_2$ are parallel lines, then for any other nonparallel line $l_0$ (as long as it is not parallel to an isotropic line), the spread between $l_1$ and $l_0$ is equal to the spread between $l_2$ and $l_0$.

\begin{proof}

\vspace{.15 in}

We will let $A \subset \bb{F}_q^2$ so that $|A| = 2q-1$. If $A$ is larger, we may simply choose a subset of size $2q-1$. Choose any spread $\theta \neq 0$ and let $\mu, \nu$ be two different lines through the origin with spread $\theta$ between them. Let $L_\mu$ be the set of lines parallel to $\mu$ and $L_\nu$ the set of lines parallel to $\nu$. Clearly, $|L_\mu| = |L_\nu| = q$. It suffices to show that there are three points $a,b,c \in A$ so that $b$ and $a$ lie on the same line in $L_\mu$ and $b$ and $c$ lie on the same line of $L_\nu$. Notice, by pigeonholing, that
$$
\#\{x \in A: |l_\mu(x) \cap A| \geq 2\} \geq q.
$$
where $l_\mu(x)$ is the line in $L_\mu$ containing the point $x$. In other words, there are at least $q$ points in $A$ that share a line of $L_\mu$ with another point of $A$. This of course applies to lines of $L_\nu$ also. By further pigeonholing, since there are only $2q-1$ points to begin with, there must be a point $b$ of $A$ which shares its $L_\mu$ line with a point $a \in A$ and its $L_\nu$ line with a point $c \in A$. Thus $(a,b,c)$ produces spread $\theta$ in $A$.
\end{proof}

\begin{remark}
Unfortunately, we had to avoid spread 0 in the above argument. In dimension $2$, three points generate a spread of $0$ only when they are collinear. If we try to apply the above argument to this case, then we get $L_\mu = L_\nu$ and the argument breaks down. We can still show that $2q-1$ points in $\bb{F}_q^2$ are enough to produce a spread of $0$, but we will do this in the next section using a different argument. Before doing so, we should note that there is an important difference between $180^\circ$ angles in Euclidean space and ``$180^\circ$ angles" (i.e. spread $0$) in vector spaces over finite fields. One can easily check that, in $\bb{F}_q^d$, three collinear points give a spread of 0 (or an undefined spread in the case that they lie on a line parallel to an isotropic line). However, in dimensions higher than $2$, three points that generate a spread of 0 need not be collinear. Take, for example, the triple $((0,0,0),(1,0,0),(2,1,2)) \in \bb{F}_5^3 \times \bb{F}_5^3 \times \bb{F}_5^3$. 
\end{remark}


\section{Collinear triples}

In this section, we will take a look at ``$180^\circ$ angles" in vector spaces over finite fields, which are just sets of three distinct points lying on a common line. We do not intend to prove anything new here; we will simply make note of some known results. A ``$k$-cap" is a set of $k$ points in a finite affine or projective space in which no three points lie on a single line. In a given space, a cap of size $k$ is called a maximal cap if there is no cap of size greater than $k$. The problem of finding maximal caps, sometimes called ``the packing problem," has been studied extensively in both affine space (see, for instance, \cite{BE} and \cite{Pot}), and projective space (see, for instance, \cite{HS} and \cite{Chao}). One of the primary motivations behind the study of caps is its application to coding theory (see \cite{Hill}). Here, we will focus primarily on the size of these maximal caps, denoted $C_d(q)$. Before we begin, note that we have been lazily using $\bb{F}_q^d$ to represent both the vector space $\bb{F}_q^d$ and the affine space $AG(d,q)$. In the interest of staying consistent with the literature on caps, we will use $AG(d,q)$ in this section.

Since corresponding affine and projective spaces only differ by a hyperplane at infinity, the sizes of maximal caps in $AG(d,q)$ and $PG(d,q)$ will be similar. That is, any cap in $AG(d,q)$ can be turned into a cap of the same size in $PG(d,q)$ via $(a,b) \rightarrow [a,b,1]$. Likewise, any maximal cap in $PG(d,q)$ will have only a relatively small portion of the points lying on the hyperplane at infinity, and disregarding those points gives us a cap in $AG(d,q)$.

The packing problem is well understood in the affine spaces $AG(d,q)$ when $d \leq 3$. In fact, we know precisely that if $q$ is odd, a maximal cap in $AG(2,q)$ has size $q+1$ and must be a conic (if $q \equiv 3 \bmod 4$, then the unit circle is one example). If $q$ is even, a maximal cap in $AG(2,q)$ has size $q+2$, but is not always defined by a conic. Similarly, a maximal cap in $AG(3,q)$ is always one point short of an elliptic quadric, and has size $q^2$ (see \cite{EFLS}).


When $d \geq 4$, the precise value of $C_d(q)$ is known for very few $q$, though some trivial bounds can be obtained. In any cap $A$, each pair of points must determine a unique line, and thus the number of pairs of points cannot exceed the number of distinct lines in $AG(d,q)$. Therefore
$$
\frac{|A|(|A|-1)}{2} \leq \frac{q^{d-1}(q^d - 1)}{q-1},
$$
which gives us $C_d(q) = O(q^{d-1})$.

One may suspect that a $(d-1)$-sphere in $AG(d,q)$ should only grant two points per line and thus be a cap of size $q^{d-1}(1+o(1))$, making our above observation sharp. Alas, this is not the case: in $AG(d,q)$, spheres of dimension greater than $3$ always contain lines. In particular, let $\vec{x}$ be any unit vector. In dimensions $4$ and higher, we can always find an isotropic vector $\vec{v}$ so that $\vec{v} \cdot \vec{x} = 0$. Then it is easy to check that the line $\{\vec{x}+t\vec{v}: t \in \bb{F}_q\}$, when viewed as a set of points, is contained in the unit sphere.

For a lower bound on $C_d(q)$, we note that if $A$ is a cap in $AG(m,q)$ and $B$ a cap in $AG(n,q)$, then $A \times B$ is a cap in $AG(m,q) \times AG(n,q) \cong AG(m+n,q)$. To show this, suppose that $A \times B$ has three points $a,b,c$ lying on a common line. Project $a$, $b$, and $c$ onto $AG(m,q)$. This projection cannot give us three distinct points, as otherwise $A$ would not be a cap in $AG(m,q)$, so all three points must project to the same point. By the same logic, their projections onto $AG(n,q)$ must be equal as well. But this is possible only if $a=b=c$. Therefore, a direct product of caps is itself a cap in the product space.

This allows us to form large caps in high-dimensional spaces out of low dimensional-caps. In particular, we can construct caps of size at least $q^{\lfloor 2d/3 \rfloor}$ in $AG(d,q)$: we take a $\lfloor d/3 \rfloor$-fold cross product of $q^2$-caps in $AG(3,q)$, and if $d \equiv 2 \bmod 3$, we may tack on a final $(q+1)$-cap in $AG(2,q)$. This tells us $C_d(q) = \Omega(q^{\lfloor 2d/3 \rfloor})$.

While maximal caps have been studied extensively, there has not been much success in improving the the exponents of the main term, i.e. $$\lfloor 2d/3 \rfloor \leq \log_q(C_d(q)) \leq d-1.$$ In \cite{HS}, Hirschfeld and Storme collected the best known bounds on maximal caps in $PG(d,q)$. While they are non-trivial, one can see that the best known upper bounds (Tables 4.4(i) and (ii)) are still on the order of $q^{d-1}$, and the best-known lower bounds (Tables 4.6(i), (ii),(iii)) are on the order of $q^{\lfloor 2d/3 \rfloor}$. There are, however, some improvements of note. Bierbrauer and Edel (\cite{BE}) applied a result of Meshulam (\cite{Mesh}) to show that $C_d(q) \leq \frac{d+1}{d^2} \cdot q^{d-1}$. More recently, Ellenberg and Gijswijt (\cite{ElGi}) used a method of Croot, Lev, and Pach (\cite{CLP}) to tackle the special case of $q = 3$. Through a clever application of the polynomial method, they showed that $C_d(3) \leq 2.756^d$, or equivalently, $C_d(3) \leq 3^{0.923d}$. While it may be possible to apply a similar method to get a result for arbitrary $q$, their proof makes use of the fact that any three collinear points in $AG(d,3)$ form an arithmetic progression, which cannot be said when $q > 3$.

\pagebreak

\section{Acknowledgements}

I would like to thank Alex Iosevich and Jonathan Pakianathan for some helpful discussions during the development of this paper.


\begin{thebibliography}{16}




\bibitem{BE} J. Bierbrauer, Y. Edel, {\it Bounds on affine caps}, J. Combin. Des. \textbf{10} (2002), no. 2, 111-–115.


\bibitem{Chao} J. Chao, {\it On the size of a cap in $PG(n,q)$ with $q$ even and $n \geq 3$}, Geom. Dedicata \textbf{74} (1999), 91--94.


\bibitem{CLP} E. Croot, V. Lev, P. Pach, {\it Progression-free sets in $\bb{Z}_n^4$ are exponentially small}, Ann. Math. \textbf{185} (2017), no. 1, 331--337.

\bibitem{EFLS} Y. Edel, S. Ferret, I, Landjev, L. Storme {\it The classification of the largest caps in $AG(5,3)$}, J. Comb. Theory A \textbf{99} (2002), 95–-110.


\bibitem{ElGi} J. Ellenberg, D. Gijswijt, {\it On large subsets of $\bb{F}_q^n$ with no three-term arithmetic progression}, Ann. Math. \textbf{185} (2017), no. 1, 339--343.

\bibitem{EF} P. Erd\H os, Z. F\" uredi, {\it The greatest angle among $n$ points in the $d$-dimensional Euclidean space}, Ann. Discrete Math. \textbf{17} (1983), 275-–283.

\bibitem{HKKMMMS} V. Harangi, T. Keleti, G. Kiss, P. Maga, A. M\'{a}th\'{e}, P. Mattila, B. Strenner, {\it How large dimension guarantees a given angle?}, Monatsh. Math. \textbf{171} (2013), no. 2, 169--187.

\bibitem{HIKR11} D. Hart, A. Iosevich, D. Koh and M. Rudnev, {\it Averages over hyperplanes, sum-product theory in finite fields, and the Erd\H os-Falconer distance conjecture}, Trans. Amer. Math. Soc. \textbf{363} (2011), no. 6, 3255--3275.

\bibitem{Hill} R. Hill, {\it Caps and Codes}, Discrete Math. \textbf{22} (1978), 111--137.

\bibitem{HS} J. Hirschfeld, L. Storme, {\it The packing problem in statistics, coding theory and finite projective spaces: update 2001}, Finite Geometries, Proceedings \textbf{3} (2001), 201–-246.


\bibitem{IR} A. Iosevich and M. Rudnev, {\it Erdos distance problem in vector spaces over finite fields}, Trans. Amer. Math. Soc. \textbf{359} (2007), no. 12, 6127--6142.

\bibitem{LPV} B. Lund, T. Pham, L. Vinh, {\it Distinct spreads in vector spaces over finite fields} (preprint 2016), last accessed Dec. 18, 2017, \url{https://arxiv.org/abs/1611.05768}.

\bibitem{Mesh} Meshulam, Roy, {\it On subsets of finite abelian groups with no 3-term arithmetic
progressions}, J. Comb. Theory A \textbf{71} (1995), no. 1, 168--172.

\bibitem{MW} D. Mubayi, J. Williford,{\it On the independence number of the Erd\" os-Renyi and projective norm graphs and a related hypergraph}, J. Graph Theory \textbf{56} (2007), no. 2, 113--127

\bibitem{Ny} J. E. Nymann, {\it On the probability that $k$ positive integers are relatively prime}, J. Number Theory \textbf{4} (1972), 469--473.

\bibitem{Pot} A. Potechin, {\it Maximal caps in $AG(6, 3)$}, Des. Codes Cryptogr. \textbf{46} (2008), 243-–259.




\end{thebibliography}
\end{document}